\documentclass[a4paper, 10.5pt]{article}
\usepackage{amsmath}
\usepackage{amssymb,esint}
\usepackage{amscd}
\usepackage{mathrsfs}
\usepackage{xspace}
\usepackage{fancyhdr}
\usepackage{xcolor}
\usepackage{verbatim}
\usepackage{cite}
\usepackage{appendix}
\usepackage{hyperref}
\usepackage{srcltx} 

\newtheorem{theorem}{Theorem}[section]

\newtheorem{lemma}[theorem]{Lemma}

\newtheorem{proposition}[theorem]{Proposition}
\newtheorem{remark}[theorem]{Remark}

\newenvironment{proof}[1][Proof]{\textbf{#1.} }{\hfill\rule{0.5em}{0.5em}}
{\catcode`\@=11\global\let\AddToReset=\@addtoreset
	\AddToReset{equation}{section}
	
	\AddToReset{theorem}{section}

\begin{document}
		\title{Well-posedness of the  Fractional Fokker--Planck Equation}
		\author{
			{\bf Ke Chen\thanks{Email address: k1chen@polyu.edu.hk, Department of Applied Mathematics, The Hong Kong Polytechnic University, Kowloon, Hong Kong, P. R. China.}},
			{\bf Ruilin Hu\thanks{Email address: huruilin16@mails.ucas.ac.cn, Academy of Mathematics and Systems Science,
					Chinese Academy of Sciences,
					Beijing, 100190, P. R. China.
				}},
				{\bf Quoc-Hung Nguyen\thanks{Email address: qhnguyen@amss.ac.cn, Academy of Mathematics and Systems Science,
						Chinese Academy of Sciences,
						Beijing, 100190, P. R. China.} }
			}
			\date{}  
			\maketitle
			
			\begin{abstract}
				In this paper, we employ a Schauder-type estimate method, as developed in \cite{CHN}, to establish critical well-posedness result for the Fractional Fokker–Planck Equation (FFPE). This equation serves as a fundamental model in kinetic theory and can be regarded as a semi-linear analogue of the non-cutoff Boltzmann equation. We demonstrate that the techniques introduced in this study are not only effective for the FFPE but also hold promise for broader applications, particularly in addressing the non-cutoff Boltzmann equation and the Landau equation.  Our results contribute to a deeper understanding of the analytical framework required for these complex kinetic models.
			\end{abstract}
			\section{Introduction}
			The Fractional Fokker-Planck Equation (FFPE) serves as a fundamental tool for describing the evolution of probability distributions in systems where anomalous diffusion occurs, which reads
			\begin{equation}\label{eq:fk}
			\begin{aligned}
			&\partial_t f+v\cdot \nabla_x f+\Lambda_v^{\alpha} f=\operatorname{div}_v(f\nabla_v \Lambda_v^{-\beta}f ),\\
			&f|_{t=0}=f_0,
			\end{aligned}
			\end{equation}
			where $f=f(t,x,v):\mathbb{R}^+\times \mathbb{R}^d\times \mathbb{R}^d\to \mathbb{R}$, $\Lambda_v=(-\Delta_v)^\frac{1}{2}$ is the fractional Laplacian operator. In this paper, we consider the regime 
			\begin{align*}
			\beta\in(0,1),\quad\quad\alpha\in(1,2),\quad \quad \alpha+\beta>2. 
			\end{align*}
			The equation is invariant under the scaling transformation
			\begin{align}\label{scale}
			f_{\lambda}(t,x,v)=\lambda^{-\kappa}f(\lambda^{-1} t,\lambda^{-(1+\frac{1}{\alpha})}x,\lambda^{-\frac{1}{\alpha}}v),
			\end{align}
			where $\kappa=\frac{\alpha+\beta-2}{\alpha}>0$.
			
			The FFPE system is the semi-linear version of non-cutoff Boltzmann system.  Consider the non-cutoff Boltzmann system
			\begin{align*}
			&\partial_tf+v\cdot\nabla_xf=\mathcal{Q}_s(f,f),\ \ s\in(0,1),\\
			&f|_{t=0}=f_0.
			\end{align*}
			Here  $\mathcal{Q}_s(f,f)$ is the Boltzmann collision operator which is defined in the following way.
			\begin{align*}
			\mathcal{Q}_s(f_1,f_2)(v):=	
			\int_{\mathbb{R}^d}\int_{S^{d-1}}(f_2(v_\star')f_1(v')-f_2(v_\star)f_1(v))B(|v_\star-v|,\cos(\theta))dS^{d-1}(\sigma) dv_\star,
			\end{align*}
			where $B$ is the standard non-cutoff collision kernel $B(|v_\star-v|,\cos(\theta))=|v_\star-v|^\gamma b(\cos(\theta))$ with 
			$b(\cos(\theta))\sim |\sin(\theta/2)|^{-(d-1)-2s}$.
			The standard relation between velocities before and after elastic collision is given by 
			\begin{align*}
			v'=\frac{v+v_\star}{2}+\frac{|v-v_\star|}{2}\sigma,~~ v'_\star=\frac{v+v_\star}{2}-\frac{|v-v_\star|}{2}\sigma.
			\end{align*}
			Using Carleman coordinates and the cancellation lemma (see \cite{ADVW}), the Boltzmann collision operator takes the following form
			\begin{align*}
			&\mathcal{Q}_s(f_1,f_2)(v)=-\mathcal{L}^s_{f_2}f_1(v)+c_bf_1(v)\int_{\mathbb{R}^d}f_2(v+w)|w|^{\gamma}dw,\quad \gamma>-d,\\
			&\mathcal{L}^s_{f_2}f_1(v)=\int_{\mathbb{R}^d}(f_1(v)-f_1(v'))\mathbf{C}_{f_2}(v,v-v')\frac{dv'}{|v-v'|^{d+2s}},\\
			&\mathbf{C}_{f_2}(v,z)=2^{d-1}\int_{w\cdot z=0}f(v+w)|w|^{2s+1+\gamma}A(\frac{|z|^2}{|w|^2})1_{|w|\geq |z|}dw,\quad A(\rho)\sim 1,
			\end{align*}
			where $A$ is a bounded function only depending on the collision kernel $B$.
			One can see that when $\mathbf{C}_{f_2}$ is replaced by $1$  and let $\alpha=2s$, then the semi-linear operator $\mathcal{L}_{f_2}^s$ reduces to $\Lambda^{\alpha}_v$, which precisely corresponds to the fractional diffusion operator in the FFPE framework. Furthermore, the nonlinear term $c_bf_1(v)\int_{\mathbb{R}^d}f_2(v+w)|w|^{\gamma}dw$ exhibits structural similarities to $\operatorname{div}_v(f\nabla_v \Lambda_v^{-\beta}f )$ for an appropriate choice of $\beta$. These fundamental parallels between the two systems provide the primary motivation for our study of the FFPE as a tractable model that captures essential features of the more complex non-cutoff Boltzmann dynamics. 
			
			The study of non-cut-off Boltzmann and Landau system is an interesting field. Local regularity for polynomially decay data is proved in \cite{CH} and the global regularity for non-cut-off Boltzmann equations can be seen in \cite{CI4}.  C. Imbert and L. Silvestre proved Schauder estimates for solution of non-cut-off Boltzmann equations in \cite{CI3}. L. He \cite{LBH} proved sharp estimates for Boltzmann and Landau collision operators. The Harnack inequality for Landau was proved in \cite{FG}, and weak Harnack inequality for non-cut-off Boltzmann equation was proved in \cite{CI2}. Other significant contributions can be seen in \cite{LBH2,GHIV}.
			
			The case $\alpha=2$ is the classical Fokker-Planck or Kolmogorov equation. \cite{AIN} proved the regularity for the solution to this system. The hypoellipticity of the operator has been proved in \cite{FB2002}. The Gevery hypoellipyicity result can be seen in \cite{CLX} and \cite{CLX2}. Readers are also recommended to \cite{CC,CV,LH} for more information.
			
			Unlike classical diffusion processes modeled by the standard Fokker-Planck equation, anomalous diffusion accounts for phenomena such as heavy-tailed distributions or long-range interactions, which are prevalent in diverse fields, including physics, biology, and finance. The FFPE incorporates fractional derivatives, representing non-local effects and memory properties, thereby providing a versatile framework to study complex stochastic dynamics.
			
			From a mathematical perspective, the FFPE can be regarded as a semilinear version of the non-cutoff Boltzmann equation, a cornerstone in kinetic theory. The study of critical well-posedness is essential in understanding the existence, uniqueness, and stability of solutions to differential equations under scaling-invariant or borderline conditions. For fluid dynamic equations such as the Navier-Stokes Equations (NSE), critical well-posedness theory has been extensively studied and has yielded profound insights into fluid behavior. F. Bouchut gived regularity results for weak solution in \cite{FB2002}, and R. Alexandre did some further estimates in \cite{RA} based on energy estimates. Hypoelliptic estimates in weighted functional spaces can be seen in Y. Morimoto and C.-J. Xu's work \cite{YMCX}. The optimal estimate was given in \cite{NL} by microlocal techiniques. More results for regularity can be seen in \cite{AIN,AIN2,HZ2}.
			
			However, extending these concepts to kinetic equations such as the FFPE and the non-cutoff Boltzmann equation presents a fundamentally different set of difficulties. These include handling the interplay between fractional diffusion operators and nonlinear terms, as well as addressing the loss of regularity and the propagation of singularities. For example, Harnack inequality fails for fractional kinetic operator, see \cite{GIHV,MKMW}. In this paper, we aim to contribute to the understanding of critical well-posedness in the context of kinetic equations by focusing on the FFPE. Building upon our previous work \cite{CHN} on a Schauder-type estimate and its applications in well-posedness theory, we develop and adapt these analytical techniques to tackle the FFPE.  By extending this methodology to the fractional and non-local operators present in the FFPE, we aim to establish new regularity results and develop a robust framework for studying the well-posedness of the equation. We mention the recent independent results of F. Grube \cite{FG}, which proved a two-sided estimate for the kernel of \eqref{eq:fk} in 1-D case by Fourier method, and the work of H. Hou and X. Zhang \cite{HouZhang}, which established a two-sided estimate for any dimension by a stochastic method. 
			
			For brevity of notations, we denote $e^{-tP}$ the semigroup associate to  the linear operator $\partial_t +P(\nabla_x,\nabla_v)$, where $P(\nabla_x,\nabla_v)$ is the differential operator with symbol $P(\xi,\eta)=(|\xi|^{\frac{2}{1+\alpha}}+|\eta|^2)^{\frac{\alpha}{2}}$. Define
			\begin{equation}\label{z1}
			[f_0]:=\sup_{t>0}t^{1+\frac{\beta-2}{\alpha}}||e^{-tP}f_0||_{L^\infty(\mathbb{R}^d_x\times \mathbb{R}^d_v)},
			\end{equation}
			
			The main result in this article is the following.    
			\begin{theorem}\label{mainthm}  Let $A\geq  100$.
				There exists $\epsilon_0>0$ such that for any initial data $f_0$ satisfying $[f_0]\leq \epsilon_0$,
				then \eqref{eq:fk} admits a unique global solution $f$ satisfying 
				\begin{align*}
				\sup_{t>0}\sum_{m+n\leq A} t^{m+\frac{m+n+\alpha+\beta-2}{\alpha}}\|\nabla_x^m\nabla_v^n f(t)\|_{L^\infty(\mathbb{R}^d_x\times \mathbb{R}^d_v)}\leq C [f_0].
				\end{align*}
			\end{theorem}
			\begin{remark} i) The $[\cdot]$ norm  defined in \eqref{z1} is scaling invariant with respect to \eqref{scale}, which implies our result is critical.\\
				ii) It is well known that for any $u=u(x)$, we have 
				\begin{equation*}
				||u||_{B^{-a}_{\infty,\infty}}\sim \sup_{t>0}t^{\frac{a}{2}}||e^{t\Delta}u||_{L^\infty(\mathbb{R}^d)}.
				\end{equation*}
				The norm in \eqref{z1} is an anisotropic Besov norm with scaling $\Lambda_v\sim \Lambda_x^{\frac{1}{1+\alpha}}$. Moreover, we have 
				\begin{align*}
				&[f_0]\leq C \|\Lambda_v^{-\gamma_1}\Lambda_x^{-\frac{\gamma_2}{(1+\alpha)}}f_0\|_{L^\infty(\mathbb{R}^d_x\times \mathbb{R}^d_v)},\\&
				[f_0]\leq C \|(\Lambda_x^{\frac{2}{1+\alpha}}+\Lambda_v^2)^{-\alpha\kappa}f_0\|_{L^\infty(\mathbb{R}^d_x\times \mathbb{R}^d_v)},
				\end{align*}
				where  $0\leq \gamma_1,\gamma_2\leq 1$ and $\gamma_1+\gamma_2=\alpha\kappa$ with $\kappa=1+\frac{\beta-2}{\alpha}$ as defined in \eqref{scale}. Here, for function $Q(|\xi|,|\eta|)$, we use the operator $Q(\Lambda_x,\Lambda_v)$ to denote $Q(\Lambda_x,\Lambda_v)f=\mathcal{F}^{-1}_{x,v}(Q(|\xi|,|\eta|)\hat{f}(\xi,\eta))$.
			\end{remark}
			
			The novelty of this work lies in the adaptation and refinement of techniques from fluid dynamics and geometric analysis to the kinetic setting, where the interplay between non-local operators and nonlinearity presents unique challenges. Specifically, we address the following key objectives:\\
			i) Establishing Critical Well-Posedness: We aim to derive conditions under which the FFPE admits unique solutions in critical spaces, ensuring that these solutions are stable and depend continuously on the initial data.\\
			ii) Extending Schauder-Type Estimates: By leveraging our previous work, we extend Schauder-type methods to fractional kinetic equations, providing tools to handle the inherent non-locality and singularities of the FFPE.\\
			iii) Bridging Gaps Between Fields: This work serves as a bridge between the well-posedness theories developed for fluid dynamic equations and their application to kinetic equations, offering new perspectives and methodologies for tackling open problems in the field.
			
			The results presented in this paper have implications beyond the FFPE itself, as the methods and insights gained here may inform the study of other kinetic equations, including the non-cutoff Boltzmann equation. Furthermore, they contribute to the broader understanding of fractional differential equations, which are increasingly recognized as powerful tools for modeling complex systems across scientific disciplines.
			
			By addressing the critical well-posedness of the FFPE, we hope to advance the mathematical theory of anomalous diffusion and kinetic equations, laying the groundwork for future research in this challenging and highly relevant area of mathematical analysis. The method in this paper can be extended to study well-posedness for non-cut-off Boltzmann equation with critical data.
			
			We will clarify some of the notations in this article. We denote $\langle a\rangle:=(1+|a|^2)^{\frac{1}{2}}$ for $a\in\mathbb{R}^n$. We denote $A\lesssim B$ if there exists universal constant $C$ such that $A\leq CB$. We denote $\delta_a^xf(x)=f(x)-f(x-a)$, and we can similarly define $\delta_a^v$. For short, we denote the Fourier transform $\mathcal{F}(f)=\hat{f}$. We denote the Lebesgue space $L^p_{x,v}=L^p(\mathbb{R}^d_x\times \mathbb{R}^d_v)$ for any $1\leq p\leq \infty$. For any function $P(\xi,\eta)$, we define the operator $P(\nabla_x,\nabla_v)f=\mathcal{F}^{-1}_{x,v}(P(\xi,\eta)\hat{f}(\xi,\eta))$.
			\section{Main theorem and proof}
			First we will prove a representation formula of the solution.
			\begin{lemma}\label{lemform}
				The solution to the Cauchy problem 
				\begin{align*}
				&\partial_t f+v\cdot \nabla_x f+\Lambda^{\alpha}_v f=F,\ \ \text{in}\ (0,+\infty)\times \mathbb{R}^d\times \mathbb{R}^d,\\
				&f|_{t=0}=f_0,
				\end{align*}
				can be written as 
				\begin{align}\label{solfor}
				f(t,x,v)=(H(t)\ast	f_0)(x-tv,v)+\int_0^t (H(t-\tau)\ast  F(\tau))(x-(t-\tau)v,v)d\tau,
				\end{align}
				where 
				\begin{align*}
				&H(t,x,v)=\frac{1}{(2\pi)^d}\int_{\mathbb{R}^{2d}}
				\exp\left(-\int_{0}^{t}|\sigma\xi-\eta|^\alpha d\sigma +i\xi\cdot x+i\eta\cdot v\right) d\xi d\eta.
				\end{align*}	
			\end{lemma}
			\begin{proof}
				We rewrite the equation in terms of Fourier transform 
				\begin{align*}
				&	\partial_t \hat f(t,\xi,\eta)-\xi\cdot\nabla_\eta \hat f(t,\xi,\eta)+|\eta|^\alpha \hat f(t,\xi,\eta)=\hat F(t,\xi,\eta),\\
				&	\hat f|_{t=0}=\hat f_0.
				\end{align*}
				Let $\hat g(t,\xi,\eta)=\hat f(t,\xi,\eta-t\xi)$, which is equivalent to $g(t,x,v)=f(t,x+tv,v)$. Then 
				\begin{align*}
				&	\partial_t \hat g(t,\xi,\eta)+|\eta-t\xi|^\alpha\hat g(t,\xi,\eta)=\hat F(t,\xi,\eta-t\xi),\\
				&	\hat g|_{t=0}=\hat f_0.
				\end{align*}
				Solving the ODE above, we obtain 
				\begin{align*}
				\hat g(t,\xi,\eta)&=\exp\left(-\int_0^t |\eta-\sigma\xi|^\alpha d\sigma\right)\hat f_0(\xi,\eta)+\int_0^t \exp\left(-\int_\tau^t |\eta-\sigma\xi|^\alpha d\sigma\right)\hat F(\tau,\xi,\eta-\tau\xi)d\tau\\
				&=\exp\left(-\int_0^t |\eta-\sigma\xi|^\alpha d\sigma\right)\hat f_0(\xi,\eta)+\int_0^t \exp\left(-\int_0^{t-\tau} |\eta-\tau\xi-\sigma\xi|^\alpha d\sigma\right)\hat F(\tau,\xi,\eta-\tau\xi)d\tau.
				\end{align*}
				Taking inverse Fourier transform, we get
				\begin{align*}
				g(t,x,v)=(H(t)\ast f_0)(x,v)+\int_0^t (H(t-\tau)\ast F(\tau))(x+\tau v,v)d\tau.
				\end{align*}
				This implies \eqref{solfor} in view of the relation $f(t,x,v)=g(t,x-tv,v)$. Then we complete the proof. \vspace{0.2cm}
			\end{proof}
			
			We have the following point-wise estimate on $H(t,x,v)$ from 	\cite{HouZhang}.
			\begin{lemma} For any $m,n\in\mathbb{N}$, it holds that
				\begin{equation}\label{z1000}
				|\nabla_x^m\nabla_v^n H(t,x,v)|\lesssim \frac{t^{-\frac{2d+m+n}{\alpha}-m-d}}{\langle t^{-\frac{1}{\alpha}-1}x, t^{-\frac{1}{\alpha}}v\rangle^{d+\alpha+1}\langle t^{-\frac{1}{\alpha}-1}\inf_{\sigma\in[0,1]}|x-\sigma tv|\rangle^{d+\alpha-1}}.
				\end{equation}
			\end{lemma}
			\begin{lemma}\label{L1nom}
				For $l_1+l_2< \alpha$, and $\gamma_1,\gamma_2,l_1,l_2\in(0,1),$ we have 
				\begin{align}
				&\||x|^{l_1}|v|^{l_2}\nabla^{j_1}_x\nabla^{j_2}_vH(t,x,v)\|_{L_{x,v}^1}\lesssim t^{\frac{(l_1-j_1)(1+\alpha)+l_2-j_2}{\alpha}},\label{e1}\\
				&\||x|^{l_1}|v|^{l_2}\delta_a^x\nabla^{j_1}_x\nabla^{j_2}_vH(t,x,v)\|_{L_{x,v}^1}\lesssim |a|^{\frac{l_1}{1+\alpha}}\min\{1,\frac{|a|}{t^{\frac{1+\alpha}{\alpha}}}\}^{1-l_1}t^{\frac{-j_1(1+\alpha)+l_2-j_2}{\alpha}},\label{e2}\\
				&\||x|^{l_1}|v|^{l_2}\delta_a^v\nabla^{j_1}_x\nabla^{j_2}_vH(t,x,v)\|_{L_{x,v}^1}\lesssim |a|^{l_2}\min\{1,\frac{|a|}{t^{\frac{1}{\alpha}}}\}^{1-l_2}t^{\frac{(l_1-j_1)(1+\alpha)-j_2}{\alpha}},\label{e3}\\
				&\|\delta_a^x\nabla^{j_1}_x\nabla^{j_2}_v\Lambda^{\gamma_1}_x\Lambda_v^{\gamma_2}H(t,x,v)\|_{L_{x,v}^1}\lesssim \min\{1,\frac{|a|}{t^{\frac{1+\alpha}{\alpha}}}\}t^{\frac{-(j_1+\gamma_1)(1+\alpha)-j_2-\gamma_2}{\alpha}},\label{e4}\\
				&\|\delta_a^v\nabla^{j_1}_x\nabla^{j_2}_v\Lambda^{\gamma_1}_x\Lambda_v^{\gamma_2}H(t,x,v)\|_{L_{x,v}^1}\lesssim \min\{1,\frac{|a|}{t^{\frac{1}{\alpha}}}\}t^{\frac{-(j_1+\gamma_1)(1+\alpha)-j_2-\gamma_2}{\alpha}}.\label{e5}
				\end{align}
			\end{lemma}
			\begin{proof} 
				The estimate \eqref{e1} follows from Remark 1.6 in \cite{HouZhang}. Then we estimate \eqref{e2}. Note that 
				\begin{align*}
				\delta_a^xH(t,x,v)=\int_0^1a\cdot \nabla_xH(t,x-\lambda a,v)d\lambda,\quad \delta_a^vH(t,x,v)=\int_0^1a\cdot \nabla_vH(t,x,v-\lambda a)d\lambda.
				\end{align*}
				For the case when $|a|\leq t^{\frac{1+\alpha}{\alpha}}$, we have
				\begin{equation}\label{e21}
				\begin{aligned}
				\||x|^{l_1}|v|^{l_2}\delta_a^x\nabla^{j_1}_x\nabla^{j_2}_vH(t,x,v)\|_{L_{x,v}^1}&\lesssim \left\||x|^{l_1}|v|^{l_2}\nabla^{j_1}_x\nabla^{j_2}_v\int_0^1a\cdot \nabla_xH(t,x-\lambda a,v)d\lambda\right\|_{L_{x,v}^1}\\
				&\lesssim \left\|\int_0^1a\cdot (|x-\lambda a|^{l_1}+|a|^{l_1})|v|^{l_2}\nabla^{j_1+1}_x\nabla^{j_2}_vH(t,x-\lambda a,v)d\lambda\right\|_{L_{x,v}^1}\\
				&\lesssim |a|t^{\frac{(l_1-j_1-1)(1+\alpha)+(l_2-j_2)}{\alpha}},
				\end{aligned}
				\end{equation}
				where we applied \eqref{e1} in the last inequality.
				For the other case $|a|\geq t^{\frac{1+\alpha}{\alpha}}$, we have
				\begin{align*}
				\||x|^{l_1}|v|^{l_2}\delta_a^x\nabla^{j_1}_x\nabla^{j_2}_vH(t,x,v)\|_{L_{x,v}^1}&\lesssim \left\||x|^{l_1}|v|^{l_2}\nabla^{j_1}_x\nabla^{j_2}_v\left(H(t,x,v)-H(t,x-a,v)\right)\right\|_{L_{x,v}^1}\\
				&\lesssim \left\|(|x|^{l_1}+|a|^{l_1})|v|^{l_2}\nabla^{j_1}_x\nabla^{j_2}_vH(t,x,v)\right\|_{L_{x,v}^1}\\
				&\lesssim |a|^{l_1}t^{\frac{-j_1(1+\alpha)+(l_2-j_2)}{\alpha}}.
				\end{align*}
				Combining this with \eqref{e21}, we obtain \eqref{e2}. The estimate \eqref{e3} follows similarly. Finally, 
				the estimates \eqref{e4} and \eqref{e5} follow from \eqref{e2}, \eqref{e3} and  the following interpolation inequalities, 
				\begin{align}\label{interpo}
				&\|\Lambda_v^\gamma f\|_{L^1_{x,v}}\lesssim \|f\|_{L^1_{x,v}}^{1-\gamma}\|\nabla_v f\|_{L^1_{x,v}}^\gamma,\quad &&\|\Lambda_x^\gamma f\|_{L^1_{x,v}}\lesssim \|f\|_{L^1_{x,v}}^{1-\gamma}\|\nabla_x f\|_{L^1_{x,v}}^\gamma,\quad &&&\gamma\in(0,1).
				\end{align}
				Note that $\Lambda^\gamma_vf(t,x,v)=c\int_{\mathbb{R}^d}\frac{\delta_b^vf(t,x,v)}{|b|^{d+\gamma}}db$, then
				\begin{equation*}
				\begin{aligned}
				\|\Lambda^\gamma_vf(x,v)\|_{L_{x,v}^1}&\lesssim\left\|\int_{\mathbb{R}^d}\frac{\delta_b^vf(x,v)}{|b|^{d+\gamma}}db\right\|_{L_{x,v}^1}\lesssim \left\|\int_{|b|\leq \lambda}\frac{\delta_b^vf(x,v)}{|b|^{d+\gamma}}db\right\|_{L_{x,v}^1}+\left\|\int_{|b|\geq \lambda}\frac{\delta_b^vf(x,v)}{|b|^{d+\gamma}}db\right\|_{L_{x,v}^1}\\
				&\lesssim\left\|\int_{|b|\leq \lambda}\frac{\int_0^1b\cdot \nabla_vf(x,v-\lambda b)d\lambda}{|b|^{d+\gamma}}db\right\|_{L_{x,v}^1}+\left\|\int_{|b|\geq \lambda}\frac{\delta_b^vf(x,v)}{|b|^{d+\gamma}}db\right\|_{L_{x,v}^1}\\
				&\lesssim \lambda^{1-\gamma}\|\nabla_vf\|_{L_{x,v}^1}+\lambda^{-\gamma}\|f\|_{L_{x,v}^1}.
				\end{aligned}
				\end{equation*}
				By taking $\lambda=\|f\|_{L_{x,v}^1}\|\nabla_vf\|_{L_{x,v}^1}^{-1}$, we get
				\begin{equation*}
				\|\Lambda_v^\gamma f\|_{L_{x,v}^1}\lesssim \|f\|_{L_{x,v}^1}^{1-\gamma}\|\nabla_vf\|_{L_{x,v}^1}^\gamma,
				\end{equation*}
				and the proof for  $\Lambda_x^\gamma$  is similar, so we have proved \eqref{interpo}. Combining \eqref{interpo} with \eqref{e2} and \eqref{e3}, we  obtain \eqref{e4} and \eqref{e5}. This completes the proof of the lemma. 
			\end{proof}\vspace{0.2cm}\\
			Note that $\alpha+\beta<3$ and $2\alpha+\beta>3$, choose $\alpha-1<\gamma<2-\beta<\alpha$. Define the norm
			\begin{align*}
			\|f\|_{X_T}=\sup_{t\in[0,T]}(t^{\kappa}\|f(t)\|_{L^\infty_{x,v}}+ \sum_{m+n\leq A} t^{\kappa +m+\frac{m+n+\gamma}{\alpha}}\|\nabla_x^m\nabla_v^n f(t)\|_{\dot C^{\frac{\gamma}{1+\alpha},\gamma}_{x,v}}),
			\end{align*}
			where $\kappa$ is fixed in \eqref{scale}, and 
			\begin{equation*}
			\|f\|_{\dot C_{v,x}^{\frac{\gamma}{1+\alpha},\gamma}}:=\sup_{a}\frac{\|\delta_a^xf\|_{L_{x,v}^\infty}}{|a|^{\frac{\gamma}{1+\alpha}}}+\sup_{b}\frac{\|\delta_b^vf\|_{L_{x,v}^\infty}}{|b|^{\gamma}}.
			\end{equation*}
			We shortly denote $X=X_\infty$. For any $g$ with $\|g\|_X\leq \epsilon$, where $\epsilon>0$ is a constant that will be fixed later in the proof, we construct a map $\mathcal{S}:g\to h$, where $h$ is the solution to the Cauchy problem 
			\begin{equation}\label{def:map}
			\begin{aligned}
			&\partial_t h+v\cdot \nabla_x h+\Lambda^{\alpha}_v h=\operatorname{div}_vF[g],\\
			&h|_{t=0}=f_0,
			\end{aligned}
			\end{equation}
			where $F[g]=g\nabla_v \Lambda^{-\beta}_vg $. By Lemma \ref{lemform},
			we have 
			\begin{equation}\label{hhh}
			\begin{aligned}
			h(t,x,v)&=\iint_{\mathbb{R}^{2d}} H(t,x-y-tv,v-u)f_0(y,u)dydu\\
			&\quad\quad+\int_0^t \iint_{\mathbb{R}^{2d}} H(t-\tau,x-(t-\tau)v,v-u) \operatorname{div}_uF[ g](\tau,y,u)dydu d\tau\\
			&=h_L(t,x,v)+h_N(t,x,v).
			\end{aligned}
			\end{equation}
			For any data $f_0$, we define the norm $[f_0]_1$ as follows
			\begin{align*}\label{def:norm}
			[f_0]_{1}:=\|h_L\|_{X}.
			\end{align*}
			\begin{lemma}\label{lemlin} There holds 
				\begin{equation*}
				[f_0]_{1}\lesssim   [f_0].
				\end{equation*}
			\end{lemma}
			\begin{proof} To complete the proof, by time rescaling, it suffices to show that
				\begin{equation}\label{z2000}
				\sum_{m+n\leq 2A} \|\nabla_x^m\nabla_v^n h_L(1)||_{L^\infty_{x,v}}\lesssim [f_0].
				\end{equation}
				From the definition of $h_L(1)$,  we can express its Fourier transform as \begin{align*}\label{z2}
				\mathcal{F}_{x,v}(h_L(1))(\xi,\eta) = \exp\left(-\int_0^1 |\eta+\tau\xi|^\alpha d\tau\right)\hat f_0(\xi,\eta) .
				\end{align*}
				Let
				\begin{equation*}
				\tilde{P}(\xi,\eta)=\langle \xi,\eta\rangle^{200N}\exp\left(-\int_0^1 |\eta+\tau\xi|^\alpha d\tau\right),
				\end{equation*}
				for $N=[\frac{10 d}{\alpha^4}]$.
				By \eqref{z1000}, we have  for any $g_0=g_0(x,v)$
				\begin{align*}
				||\nabla_{x,v}^n\tilde{P}(\nabla_x,\nabla_v) g_0||_{L^\infty_{x,v}}\lesssim ||g_0||_{L^\infty_{x,v}}~~\forall n\in \mathbb{N}.
				\end{align*}
				Using the Fourier representation of $h_L(1)$, we obtain \begin{align}\label{s1}
				\sum_{m+n\leq 2A} \|\nabla_x^m\nabla_v^n h_L(1)||_{L^\infty_{x,v}}\lesssim ||\langle \nabla_x,\nabla_v\rangle^{-200N}f_0||_{L^\infty_{x,v}}.
				\end{align}
				Since \begin{align*}
				\int_0^\infty t^{[\frac{4}{\alpha}]-1}e^{-t}e^{-at }dt=c(1+a)^{-[\frac{4}{\alpha}]},\end{align*}
				let $P$ be the operator defined in \eqref{z1}, we get 
				\begin{align*}
				\|(1+P)^{-[\frac{4}{\alpha}]}f_0\|_{L^\infty_{x,v}}=c^{-1}\left\| \int_0^\infty t^{[\frac{4}{\alpha}]-1}e^{-(1+P)t }f_0dt\right\|_{L^\infty_{x,v}}\lesssim [f_0].
				\end{align*}
				This leads to  
				\begin{align*}
				||\langle \nabla_x,\nabla_v\rangle^{-200 N}f_0||_{L^\infty_{x,v}}\lesssim [f_0].
				\end{align*}
				Combining this with \eqref{s1}, we obtain \eqref{z2000}.
				The proof is complete.
			\end{proof}\vspace{0.3cm}\\
			Denote $\gamma_0=\gamma+1-\alpha$, we have $0<\gamma_0<\min\{1,\gamma\}$.  We prove the following lemma to control the nonlinear term  $h_N$.
			\begin{lemma}\label{lemnl}
				Denote 
				\begin{align*}
				\mathcal{M}F(t,x,v)=\int_0^t \iint_{\mathbb{R}^{2d}}\nabla_uH(t-\tau,x-y-(t-\tau)v,v-u)\cdot F(\tau,y,u)dydud\tau.
				\end{align*}
				Then 
				\begin{align*}
				\|\mathcal{M}F\|_{X}\lesssim |||F|||,
				\end{align*}
				where 
				\begin{align*}
				|||F|||:=\sup_{t>0}\sum_{m+n\leq A}t^{\kappa+1+\frac{m(1+\alpha)+n-1+\gamma_0}{\alpha}}\|\nabla^m_x\nabla^n_vF(t)\|_{\dot C^{\frac{\gamma_0}{1+\alpha},\gamma_0}_{x,v}	}.
				\end{align*}
			\end{lemma}
			\begin{proof} Note that 
				\begin{align*}
				\iint_{\mathbb{R}^{2d}}\nabla_uH(t-\tau,x-y-(t-\tau)v,v-u)dydu=0,
				\end{align*}
				hence, we have 
				\begin{align*}
				\mathcal{M}F(t,x,v)=\int_0^t \iint_{\mathbb{R}^{2d}}\nabla_uH(t-\tau,x-y-(t-\tau)v,v-u)\cdot (F(\tau,y,u)-F(\tau,y,v))dydud\tau.
				\end{align*}
				From this and Lemma \ref{L1nom}, we obtain 
				\begin{align*}
				\|\mathcal{M}F(t)\|_{L^\infty} &\lesssim \int_0^t \||u|^{\gamma_0} \nabla_u H(t-\tau, y, u)\|_{L_{u,y}^1} \|F(\tau)\|_{L_y^\infty \dot{C}_u^{\gamma_0}} \, d\tau\\
				&\lesssim |||F||| \int_0^t (t-\tau)^{\frac{\gamma_0-1}{\alpha}} \tau^{-\kappa - 1 + \frac{1-\gamma_0}{\alpha}} \, d\tau.
				\end{align*}
				Since \( -\kappa - 1 + \frac{1-\gamma_0}{\alpha} = -1 + \frac{2-\beta-\gamma}{\alpha} > -1 \), the integral converges, leading to:
				\begin{equation}\label{lem61}
				\|\mathcal{M}F(t)\|_{L^\infty} \lesssim t^{-\kappa} |||F|||.
				\end{equation}
				Then we consider the higher order Hölder norms, we first analyze \(\delta_a^x \nabla_x^{m_1} \nabla_v^{m_2} \mathcal{M}F\). Using the definition of \(\mathcal{M}F\), we write:
				\[
				\delta_a^x \nabla_x^{m_1} \nabla_v^{m_2} \mathcal{M}F = (-1)^{m_2} \int_0^t \iint_{\mathbb{R}^{2d}} \sum_{l_1+l_2=m_2} (t-\tau)^{l_1} \delta_a^x \nabla_x^{m_1+l_1} \nabla_u^{l_2} H(t-\tau, x-(t-\tau)v-y, v-u) \operatorname{div}_u F(\tau, y, u) \, dy \, du \, d\tau.
				\]
				Similarly, for \(\delta_a^v \nabla_x^{m_1} \nabla_v^{m_2} \mathcal{M}F\), we have:
				\[
				\begin{aligned}
				&\delta_a^v \nabla_x^{m_1} \nabla_v^{m_2} \mathcal{M}F \\
				&\quad = (-1)^{m_2} \int_0^t \iint_{\mathbb{R}^{2d}} \sum_{l_1+l_2=m_2} (t-\tau)^{l_1} \delta_{-a}^u \nabla_x^{m_1+l_1} \nabla_u^{l_2} H(t-\tau, x-(t-\tau)(v-a)-y, v-u) \operatorname{div}_u F(\tau, y, u) \, dy \, du \, d\tau \\
				&\quad + (-1)^{m_2} \int_0^t \iint_{\mathbb{R}^{2d}} \sum_{l_1+l_2=m_2} (t-\tau)^{l_1} \delta_{-(t-\tau)a}^y \nabla_x^{m_1+l_1} \nabla_u^{l_2} H(t-\tau, x-(t-\tau)v-y, v-u) \operatorname{div}_u F(\tau, y, u) \, dy \, du \, d\tau.
				\end{aligned}
				\]
				1) Hölder Estimate in \(x\):\\
				For the \(x\)-Hölder norm, we estimate:
				\begin{equation}\label{lem62}
				\|\delta_a^x \nabla_x^{m_1} \nabla_v^{m_2} \mathcal{M}F(t)\|_{L^\infty} \lesssim I_1 + I_2,
				\end{equation}
				where
				\[
				I_1 = \int_0^{\frac{t}{2}} \sum_{l_1+l_2=m_2} \left\||u|^{\gamma_0} (t-\tau)^{l_1} \nabla_y^{m_1+l_1} \nabla_u^{l_2} \nabla_u \delta_{-a}^y H(t-\tau, y, u)\right\|_{L_{y,u}^1} \|F(\tau)\|_{L_y^\infty \dot{C}_u^{\gamma_0}} \, d\tau,
				\]
				and
				\[
				I_2 = \int_{\frac{t}{2}}^t \sum_{l_1+l_2=m_2} \left\||u|^{\gamma_0} (t-\tau)^{l_1} \delta_{-a}^y \nabla_u H(t-\tau, y, u)\right\|_{L_{y,u}^1} \|\nabla_u^{l_2} \nabla_y^{m_1+l_1} F(\tau)\|_{L_y^\infty \dot{C}_u^{\gamma_0}} \, d\tau.
				\]
				Applying the kernel estimates in Lemma \ref{L1nom}, we deduce:
				\[
				\|\delta_a^x \nabla_x^{m_1} \nabla_v^{m_2} \mathcal{M}F(t)\|_{L^\infty} \lesssim |a|^{\frac{\gamma}{1+\alpha}} t^{-\frac{(1+\alpha)m_1 + m_2 + \gamma}{\alpha} - \kappa} |||F|||. 
				\]
				2) Hölder Estimate in \(v\): \\
				For the \(v\)-Hölder norm, the estimate follows similarly.  We have 
				\begin{equation*}
				\|\delta_a^v\nabla_x^{m_1}\nabla_v^{m_2}\mathcal{M}F(t)\|_{L^\infty}\lesssim \sum_{i=1}^4II_i,
				\end{equation*}
				with
				\begin{align*}
				&II_1=\int_0^\frac{t}{2}\sum_{l_1+l_2=m_2}\left\||u|^{\gamma_0}(t-\tau)^{l_1}\nabla_y^{m_1+l_1}\nabla_u^{l_2}\nabla_u\delta_{-(t-\tau)a}^yH(t-\tau,y,u)\right\|_{L_{y,u}^1}\|F(\tau)\|_{L_y^\infty \dot C^{\gamma_0}_u}d\tau,\\
				& II_2=\int_0^\frac{t}{2}\sum_{l_1+l_2=m_2}\left\||u|^{\gamma_0}(t-\tau)^{l_1}\nabla_y^{m_1+l_1}\nabla_u^{l_2}\nabla_u\delta_{-a}^uH(t-\tau,y,u)\right\|_{L_{y,u}^1}\|F(\tau)\|_{L_y^\infty\dot C_u^{\gamma_0}}d\tau,\\
				& II_3=\int_\frac{t}{2}^t\sum_{l_1+l_2=m_2}\left\||u|^{\gamma_0}(t-\tau)^{l_1}\delta_{-(t-\tau)a}^y\nabla_uH(t-\tau,y,u)\right\|_{L_{y,u}^1}\|\nabla_y^{m_1+l_1}\nabla_u^{l_2}F(\tau)\|_{L_y^\infty\dot C_{u}^{\gamma_0}}d\tau,\\
				&II_4=\int_\frac{t}{2}^t\sum_{l_1+l_2=m_2}\left\||u|^{\gamma_0}(t-\tau)^{l_1}\delta_{-a}^u\nabla_uH(t-\tau,y,u)\right\|_{L_{y,u}^1}\|\nabla_y^{m_1+l_1}\nabla_u^{l_2}F(\tau)\|_{L_y^\infty\dot C_{u}^{\gamma_0}}d\tau.
				\end{align*}
				
				By kernel estimates Lemma \ref{L1nom}, we can obtain
				\begin{equation}\label{lem63}
				\|\delta_a^v \nabla_x^{m_1} \nabla_v^{m_2} \mathcal{M}F(t)\|_{L^\infty} \lesssim |a|^\gamma t^{-\frac{(1+\alpha)m_1 + m_2 + \gamma}{\alpha} - \kappa} |||F|||.
				\end{equation}
				Then \eqref{lem61}, \eqref{lem62} and \eqref{lem63} conclude the proof.
			\end{proof}
			
			\begin{lemma}\label{estforce}
				For any $f,g\in X$, we have the estimate
				\begin{align}\label{rrr}
				|||f_2 \nabla_v \Lambda_v^{-\beta} f_1||| \lesssim \|f_1\|_{X} \|f_2\|_{X}.
				\end{align}
			\end{lemma}
			\begin{proof}
				Note that for any $0<\mu<1$,
				\begin{equation}\label{interpolation}
				\begin{aligned}
				&\|fg\|_{\dot C_*^{\mu}}\lesssim\|f\|_{L_*^\infty}\|g\|_{\dot C_*^{\mu}}+\|f\|_{\dot C_*^{\mu}}\|g\|_{L_*^\infty},\\
				&\|\Lambda_*^{\mu}f\|_{L^\infty}\lesssim \|f\|_{\dot C_*^{\mu-\epsilon}}^{\frac{1}{2}}\|f\|_{\dot C_*^{\mu+\epsilon}}^{\frac{1}{2}},\quad 0<\mu+\epsilon,\mu-\epsilon<1,
				\end{aligned}
				\end{equation}
				with $*\in\{x,v\}$. Note that the Riesz transform is  $\dot C^{a}$ bounded for any $0<a<1$. Furthermore, we claim that 
				\begin{equation}\label{est:Holder}
				\begin{aligned}
				\|\Lambda_v^{1-\beta}\nabla_x^m\nabla_v^nf\|_{L_v^\infty\dot C_x^{\frac{\gamma_0}{1+\alpha}}}+\|\Lambda_v^{1-\beta}\nabla_x^m\nabla_v^nf\|_{L_x^\infty\dot C_v^{\gamma_0}}\lesssim t^{-\frac{(1+\alpha)m+n+\gamma}{\alpha}}\|f\|_{X}.
				\end{aligned}
				\end{equation}
				In fact, first for H\"{o}lder norm of $x$, for the case $|b|\geq t^{\frac{1+\alpha}{\alpha}}$,
				\begin{equation*}
				\frac{\|\delta_b^x\Lambda_v^{1-\beta}\nabla_x^m\nabla_v^nf\|_{L_{x,v}^\infty}}{|b|^{\frac{\gamma_0}{1+\alpha}}}\lesssim t^{-\frac{\gamma_0}{\alpha}}\|\Lambda_v^{1-\beta}\nabla_x^m\nabla_v^nf\|_{L_{x,v}^\infty}\lesssim t^{-\frac{\gamma+(1+\alpha)m+n}{\alpha}}\|f\|_{X}.
				\end{equation*}
				For the case $|b|\leq t^{\frac{1+\alpha}{\alpha}}$, we split the integral into two parts,
				\begin{align*}
				\delta_b^xD^{1-\beta}_v\nabla_x^m\nabla_v^nf&=C_\beta \int_{\mathbb{R}^d}\frac{\delta_b^x\delta_a^v\nabla_x^m\nabla_v^nf}{|a|^{d+1-\beta}}da\\
				&=C_\beta \int_{|a|<|b|^{\frac{1}{1+\alpha}}}\frac{\delta_b^x\delta_a^v\nabla_x^m\nabla_v^nf}{|a|^{d+1-\beta}}da+C_\beta \int_{|a|>|b|^{\frac{1}{1+\alpha}}}\frac{\delta_b^x\delta_a^v\nabla_x^m\nabla_v^nf}{|a|^{d+1-\beta}}da\\
				&:=I_1+I_2.
				\end{align*}
				We have
				\begin{align*}
				|I_1|&\lesssim \int_{|a|< |b|^{\frac{1}{1+\alpha}}}|a|^{-d-1+\beta+\gamma}da\|\nabla^m_x\nabla^n_vf(t)\|_{L_x^\infty\dot C_v^\gamma}
				\lesssim |b|^{\frac{\gamma_0}{1+\alpha}}t^{-\frac{\gamma+(1+\alpha)m+n}{\alpha}}\|f\|_{X}.
				\end{align*} 
				And 
				\begin{equation*}
				|I_2|\lesssim \int_{|a|> |b|^{\frac{1}{1+\alpha}}}|a|^{-d-1+\beta}da|b|^{\frac{\gamma}{1+\alpha}}\|\nabla^m_x\nabla^n_vf(t)\|_{L_{v}^\infty\dot C_x^{\frac{\gamma}{1+\alpha}}}\lesssim |b|^{\frac{\gamma_0}{1+\alpha}}t^{-\frac{\gamma+(1+\alpha)m+n}{\alpha}}\|f\|_{X}.
				\end{equation*}
				Combining the above estimates, we obtain 
				\begin{align*}
				\|\Lambda_v^{1-\beta}\nabla_x^m\nabla_v^nf(t)\|_{L_v^\infty\dot C_x^{\frac{\gamma_0}{1+\alpha}}}\lesssim t^{-\frac{(1+\alpha)m+n+\gamma}{\alpha}}\|f\|_{X}.
				\end{align*}
				Similarly, we have 
				\begin{align*}
				\|\Lambda_v^{1-\beta}\nabla_x^m\nabla_v^nf(t)\|_{L_x^\infty\dot C_v^{\gamma_0}}\lesssim t^{-\frac{(1+\alpha)m+n+\gamma}{\alpha}}\|f\|_{X}.
				\end{align*}
				Thus, we obtain  \eqref{est:Holder}.\\
				Note that $1-\beta+\gamma_0\leq \gamma$, let $F=f_2\nabla_v\Lambda_v^{-\beta}f_1$. By \eqref{interpolation} and \eqref{est:Holder}, for any $m,n\in\mathbb{N}$, $m+n\leq A$, we have
				\begin{align*}
				&\|\nabla_x^{m}\nabla_v^{n}F(t)\|_{\dot C_{x,v}^{\frac{\gamma_0}{1+\alpha},\gamma_0}}\\
				&\lesssim\sum_{\substack{m_1+m_2=m\\ n_1+n_2=n}}\|\nabla_x^{m_1}\nabla_v^{n_1}\nabla_v\Lambda^{-\beta}_vf_1(t)\|_{\dot C_{x,v}^{\frac{\gamma_0}{1+\alpha},\gamma_0}}\|\nabla_x^{m_2}\nabla_v^{n_2}f_2(t)\|_{L^\infty_{x,v}}\\
				&\quad\quad+\sum_{\substack{m_1+m_2=m\\ n_1+n_2=n}}\|\nabla_x^{m_1}\nabla_v^{n_1}\nabla_v\Lambda^{-\beta}_vf_1(t)\|_{L^\infty_{x,v}}\|\nabla_x^{m_2}\nabla_v^{n_2}f_2(t)\|_{\dot C_{x,v}^{\frac{\gamma_0}{1+\alpha},\gamma_0}}\\
				&\lesssim t^{-\kappa-1-\frac{(1+\alpha)m+n+\gamma_0-1}{\alpha}}\|f_1\|_{X}\|f_2\|_{X}.
				\end{align*}
				This implies \eqref{rrr} and completes the proof of the lemma.
			\end{proof}\\
			
			For \( g \in X \), let \( h = \mathcal{S}(g) \) be the map defined in \eqref{def:map}. We now prove that there exists $\sigma>0$ such that  \( \mathcal{S} \) is a contraction map in the set $$
			X^\sigma:= \{f:\|f\|_X\leq \sigma\}.
			$$
			
			\begin{proposition}\label{contrac}
				For any $f_0$ with  \( [f_0] <\infty\), and any $g,g_1,g_2\in X$, it holds
				\begin{align}
				&\|\mathcal{S}g\|_{X} \leq C_1([f_0] + \|g\|_{X}^2),\label{map1}\\
				&	\|\mathcal{S}g_1 - \mathcal{S}g_2\|_{X} \leq C_1(\|g_1\|_{X} + \|g_2\|_{X}) \|g_1 - g_2\|_{X}.\label{map2}
				\end{align}
			\end{proposition}
			\begin{proof}
				Recall from \eqref{hhh} that $h=h_L+h_N$.
				By Lemma \ref{lemlin}, Lemma \ref{lemnl} and Lemma~\ref{estforce}, we have
				\[
				\|h_L\|_{X} \lesssim [f_0], \quad \|h_N\|_{X} \lesssim \|g\|_{X}^2.
				\]
				Hence, we obtain \eqref{map1}.
				
				For the contraction property, let \( \mathbf{h} = h_1 - h_2 = \mathcal{S}g_1 - \mathcal{S}g_2 \) and \( \mathbf{g} = g_1 - g_2 \). Then \( \mathbf{f} \) satisfies
				\[
				\begin{cases}
				\partial_t \mathbf{h} + v \cdot \nabla_x \mathbf{h} + \Lambda_v^\alpha \mathbf{h} = \operatorname{div}(F[g_1] - F[g_2]), \\
				\mathbf{h}|_{t=0} = 0.
				\end{cases}
				\]
				By Lemma~\ref{lemnl}, we have
				\[
				\|\mathbf{h}\|_{X} \lesssim |||F[g_1] - F[g_2]|||.
				\]
				Since
				\[
				F[g_1] - F[g_2] = g_1 \nabla_v \Lambda_v^{-\beta} \mathbf{g} + \mathbf{g} \nabla_v \Lambda_v^{-\beta} g_2,
				\]
				Lemma~\ref{estforce} implies
				\[
				|||F[g_1] - F[g_2]||| \lesssim \|g_1 - g_2\|_{X} (\|g_1\|_{X} + \|g_2\|_{X}).
				\]
				This yields \eqref{map2} and completes the proof.
			\end{proof}\vspace{0.3cm}\\
			\begin{proof}[Proof of Theorem~\ref{mainthm}]
				We show that \( \mathcal{S} \) is a contraction map. Take $\epsilon_0=\frac{1}{100(C_1+1)}$ and $\sigma=2C_1[f_0]$. By Proposition~\ref{contrac}, for any $f_0$ such that $[f_0]\leq \epsilon_0$, and any  $g,g_1,g_2\in X^\sigma$,
				it holds                \begin{align*}
				&\|\mathcal{S}g\|_X\leq C_1([f_0]+\sigma^2)\leq \sigma,\\
				&\|\mathcal{S}g_1-\mathcal{S}g_2\|_X\leq C_1\sigma\|g_1-g_2\|_X\leq \frac{1}{2}\|g_1-g_2\|_X.
				\end{align*}
				Hence $\mathcal{S}$ is a contraction map in $X^\sigma$. By the contraction mapping theorem, there exists a unique $f\in X^\sigma$ such that $f=\mathcal{S}f$, which is a solution to the equation \eqref{eq:fk}. This completes the proof.
			\end{proof}
			\section* {Acknowledgments}
			
			Q. H.  Nguyen  is supported by the Academy of Mathematics and Systems Science, Chinese Academy of Sciences startup fund, and the National Natural Science Foundation of China (No. 12050410257 and No. 12288201) and  the National Key R$\&$D Program of China under grant 2021YFA1000800. 
			
		\end{document}